\documentclass[conference]{ieeeconf}
\IEEEoverridecommandlockouts
\usepackage{cite}
\usepackage{amsmath,amssymb,amsfonts}
\usepackage{multirow}
\usepackage{graphicx}
\usepackage{textcomp}
\usepackage{epsfig}
\usepackage[noend]{algorithmic}
\usepackage{cleveref}[2012/02/15]
\usepackage{epstopdf}
\usepackage[ruled]{algorithm}
\usepackage{enumerate}
\usepackage{float}
\usepackage{graphics,graphicx}
\usepackage{subfigure}
\usepackage{array}
\usepackage[justification=centering]{caption}
\usepackage{footnote}
\usepackage{lipsum}
\makeatletter
\usepackage{color}
\usepackage{xspace}

\usepackage{footnote}

\newcommand{\mb}[1]{\mathbf{#1}}
\newcommand{\mbb}[1]{\mathbb{#1}}
\newcommand{\mc}[1]{\mathcal{#1}}

\def\T{\operatorname{T}}

\def\N{\mathbb{N}}

\def\bx{\mathbf{x}}
\def\bz{\mathbf{z}}
\def\by{\mathbf{y}}

\def\BibTeX{{\rm B\kern-.05em{\sc i\kern-.025em b}\kern-.08em
    T\kern-.1667em\lower.7ex\hbox{E}\kern-.125emX}}

\crefformat{footnote}{#2\footnotemark[#1]#3}

\newtheorem{lemma}{Lemma}
\newtheorem{theorem}{Theorem}

\newtheorem{assumption}{Assumption}

\begin{document}
	
\title{\LARGE\bf Delay-agnostic Asynchronous Distributed Optimization}
\author{Xuyang Wu, Changxin Liu, Sindri Magn\'{u}sson, and Mikael Johansson\thanks{X. Wu, C. Liu, and M. Johansson are with the Division of Decision and Control Systems, School of EECS, KTH Royal Institute of Technology, SE-100 44 Stockholm, Sweden. Email: {\tt {\{xuyangw,changxin,mikaelj\}@kth.se}.}}
\thanks{S. Magn\'{u}sson is with the Department of Computer and System Science, Stockholm University, SE-164 07 Stockholm, Sweden. Email: {\tt sindri.magnusson@dsv.su.se}.} \thanks{This work was supported by WASP and the Swedish Research Council (Vetenskapsr\r{a}det) under grants 2019-05319 and 2020-03607. }}
\maketitle
\begin{abstract}
Existing asynchronous distributed optimization algorithms often use diminishing step-sizes that cause slow practical convergence, or fixed step-sizes that depend on an assumed upper bound of delays. Not only is such a delay bound hard to obtain in advance, but it is also large and therefore results in unnecessarily slow convergence. This paper develops asynchronous versions of two distributed algorithms, DGD and DGD-ATC, for solving consensus optimization problems over undirected networks. In contrast to alternatives, our algorithms can converge to the fixed point set of their synchronous counterparts using step-sizes that are independent of the delays. We establish convergence guarantees under both partial and total asynchrony. The practical performance of our algorithms is demonstrated by numerical experiments.
\end{abstract}


\section{Introduction}\label{sec:intro}

Distributed optimization has attracted much attention in the last decade and has found applications in diverse areas such as cooperative control, machine learning, and power systems. 
The literature on distributed optimization has primarily focused on synchronous methods that iterate in a serialized manner, proceeding to the next iteration only after the current one is completed. Synchronous methods also require all nodes to maintain a consistent view of optimization variables without any information delay, which makes the algorithms easier to analyze. Nevertheless, synchronization through a network can be challenging. Additionally, synchronized update is inefficient and unreliable since the time taken per iteration is determined by the slowest node and the optimization process is vulnerable to single-node failure. 

Asynchronous distributed methods that do not require synchronization between nodes are often better suited for practical implementation \cite{assran2020advances}. However, asynchronous methods are subject to information delays and nodes do not have a consistent view of the optimization variables, which makes them difficult to analyze. Despite the inherent challenges, there have been notable successes in studying the mathematical properties of asynchronous optimization algorithms. One area of focus has been on asynchronous consensus optimization algorithms \cite{nedic2010convergence,zhang2019asyspa,sirb2016consensus,doan2017convergence,kungurtsev2023decentralized,assran2020asynchronous,zhang2019fully,wu2017decentralized,tian2020achieving}, including asynchronous variants of well-established consensus optimization algorithms such as DGD, PG-EXTRA, and gradient-tracking-based methods. Asynchronous distributed algorithms on other optimization problems include ADGD \cite{wang2021asynchronous}, Asy-FLEXA \cite{cannelli2020asynchronous}, the asynchronous primal-dual algorithm \cite{latafat2022primal}, and the asynchronous coordinate descent method \cite{ubl2022faster,ubl2021totally}. 

The above work mainly focused on two types of step-size strategies: diminishing step-sizes~\cite{zhang2019asyspa,sirb2016consensus,doan2017convergence,kungurtsev2023decentralized,assran2020asynchronous} and fixed delay-dependent step-sizes~\cite{wu2017decentralized, tian2020achieving, cannelli2020asynchronous, latafat2022primal, zhang2019fully,ubl2022faster}. While diminishing step-sizes are effective in stochastic optimization or non-smooth optimization, they can result in slow convergence rates in deterministic smooth problems. For these types of problems, faster algorithms can often be obtained with non-diminishing step-sizes. Fixed step-sizes that depend on delay, in contrast, usually require an upper bound on the worst-case delay that is challenging to compute prior to executing the algorithm. Moreover, the use of worst-case delay can result in a conservative step-size condition and consequently, slow down the practical convergence speed. This is because the actual delays experienced in practice may be significantly smaller than the worst-case delay. For example, \cite{mishchenko2022asynchronous} implements an asynchronous SGD on a 40-core CPU, and reports a maximum and average delay of around $1200$ and $40$, respectively. Convergence of asynchronous distributed algorithms with fixed step-sizes that do not include any delay information have been considered in \cite{nedic2010convergence,ubl2021totally,wang2021asynchronous}. However, \cite{nedic2010convergence,ubl2021totally} only consider quadratic programming and \cite{wang2021asynchronous} studies only star networks.

In this paper, we study the asynchronous variants of two distributed algorithms, the decentralized gradient descent (DGD) \cite{yuan2016convergence} and the DGD using the adapt-then-combine technique (DGD-ATC) \cite{pu2020asymptotic},  for solving consensus optimization over undirected networks. Our contributions include:
\begin{enumerate}
    \item We establish the optimality gap between the fixed point of DGD-ATC with fixed step-sizes and the optimum of the consensus optimization problem. This result is absent in the literature. 
    \item We show theoretically that, under the total asynchrony assumption, the two asynchronous methods can converge to the same fixed point sets of their synchronous counterparts with \emph{fixed step-sizes that do not include delay information}.
    \item We improve the above asymptotic convergence to linear convergence by assuming bounded information delays.
\end{enumerate}
Compared to the delay-dependent fixed step-sizes, our proposed delay-free step-sizes are easy to tune and, in general, less restrictive. Although algorithms that use delay-dependent fixed step-sizes \cite{wu2017decentralized, zhang2019fully,tian2020achieving,cannelli2020asynchronous, latafat2022primal, ubl2022faster} or diminishing step-sizes \cite{zhang2019asyspa,assran2020asynchronous,sirb2016consensus,doan2017convergence,kungurtsev2023decentralized} can theoretically converge to the optimum while our algorithms suffer from unfavourable inexact convergence inherited from their synchronous counterparts, our algorithms may achieve faster practical convergence due to their less restrictive fixed step-sizes, which is demonstrated by numerical experiments.


The outline of this paper is as follows: Section II formulates the problem, revisits the synchronous algorithms DGD and DGD-ATC, and reviews/establishes their optimality error bounds. Section III introduces the asynchronous DGD and the asynchronous DGD-ATC, and Section IV provides convergence results. Finally, Section V tests the practical performance of the two asynchronous algorithms by numerical experiments and Section VI concludes the paper.

\subsection*{Notation and Preliminaries}

We use $\mb{1}_d$, $\mathbf{0}_{d\times d}$, and $I_d$ to denote the $d$-dimensional all-one vector, the $d\times d$ all-zero matrix, and the $d\times d$ identity matrix, respectively, where the subscript is omitted when it is clear from context. The notation $\otimes$ represents the Kronecker product and $\N_0$ is the set of natural numbers including $0$. For any symmetric matrix $W\in\mbb{R}^{n\times n}$, $\lambda_i(W)$, $1\le i\le n$ denotes the $i$th largest eigenvalue of $W$, $\operatorname{Range}(W)$ is its range, and $W\succ \mathbf{0}$ means that $W$ is positive definite. For any vector $x\in\mathbb{R}^n$, we use $\|x\|$ to represent the $\ell_2$ norm and define $\|x\|_W=\sqrt{x^TWx}$ for any positive definite matrix $W\in\mathbb{R}^{n\times n}$. For any differentiable function $f:\mathbb{R}^d\rightarrow\mathbb{R}$, we say it is $L$-smooth for some $L>0$ if 
\begin{equation*}
    \|\nabla f(y)-\nabla f(x)\|\le L\|y-x\|,~\forall x,y\in\mathbb{R}^d
\end{equation*}
and it is $\mu$-strongly convex for some $\mu>0$ if \begin{equation*}
    \langle \nabla f(y)-\nabla f(x), y-x\rangle\ge \mu\|y-x\|^2,~\forall x,y\in\mathbb{R}^d.
\end{equation*}

\section{Problem Formulation and Synchronous distributed Algorithms}

This section describes consensus optimization and revisits the synchronous distributed algorithms, DGD \cite{yuan2016convergence} and DGD-ATC \cite{pu2020asymptotic}, for solving it. The asynchronous version of the two methods will be introduced in Section \ref{sec:asycalg}.

\subsection{Consensus Optimization}

Consider a network of $n$ agents described by an undirected, connected graph 
$\mc{G}=(\mc{V}, \mc{E})$, where $\mc{V}=\{1,\ldots,n\}$ is the vertex set and $\mc{E}\subseteq \mc{V}\times \mc{V}$ is the edge set. In the network, each agent $i$ observes a local cost function $f_i:\mathbb{R}^d\rightarrow \mathbb{R}$ and can only interact with its neighbors in $\mc{N}_i=\{j: \{i,j\}\in\mc{E}\}$. Consensus optimization aims to find a common decision vector that minimizes the total cost of all agents:
\begin{equation}\label{eq:consensusprob}
\begin{split}
	\underset{x\in\mathbb{R}^d}{\operatorname{minimize}}~&~f(x)=\sum_{i\in\mc{V}} f_i(x).
\end{split}
\end{equation}

Distributed algorithms for solving Problem \eqref{eq:consensusprob} include the distributed subgradient method \cite{nedic2009distributed}, DGD \cite{yuan2016convergence}, distributed gradient-tracking-based algorithm \cite{nedic17}, distributed dual averaging\cite{liu2022decentralized}, and PG-EXTRA \cite{shi2015proximal}. While these algorithms were originally designed to be executed synchronously, they have since been extended to allow for asynchronous implementations. However, existing asynchronous methods often suffer from slow convergence due to the use of either diminishing step-sizes or fixed step-sizes that depend on a (usually unknown and large) upper bound on all delays. 

In this paper, we analyse the asynchronous version of   two
algorithms with delay-free fixed step-sizes: Decentralized Gradient Descent (DGD) and DGD using Adapt-Then-Combine Technique (DGD-ATC).

%
%

\subsection{Decentralized Gradient Descent (DGD)}

 The first algorithm is DGD~\cite{yuan2016convergence}. To present the algorithm compactly, define
$\mathbf{x}=(x_1^T, \ldots, x_n^T)^T\in\mathbb{R}^{nd}$, $F(\bx) = \sum_{i\in\mc{V}} f_i(x_i)$, 
and  let
$\mb{W}=W\otimes I_d$ where
$W$ is an averaging matrix\footnote{We say a matrix $W=(w_{ij})\in\mbb{R}^{n\times n}$ is an averaging matrix associated with $\mc{G}$ if it is non-negative, symmetric ($W=W^T$), stochastic ($W\mb{1}=\mb{1}$), and satisfies $w_{ij}=0$ if and only if $\{i,j\}\notin \mc{E}$ and $i\ne j$. This matrix can be easily formed in a distributed manner, with many options listed in \cite[Section 2.4]{shi2015extra}.} associated with $\mc{G}$. We use $k\in\mathbb{N}_0$ as iteration index and $\mathbf{x}^k$ as the value of $\bx$ at iteration $k$. Then the DGD algorithm progresses according to the following iterations: 
\begin{equation}\label{eq:DGD}
    \mathbf{x}^{k+1}= \mb{W}\bx^k - \alpha \nabla F(\bx^k),
\end{equation}
where $\alpha>0$ is the step-size. 

 As shown in~\cite{yuan2016convergence}, the DGD algorithm converges to a fixed point under reasonable assumptions. However, while the set of fixed points of DGD is not identical to the set of optimal solution of Problem \eqref{eq:consensusprob}, it is possible to bound the difference between the two sets under the following assumptions:
%
\begin{assumption}\label{asm:prob}
Each $f_i$ is proper closed convex, lower bounded, and $L_i$-smooth for some $L_i>0$. Further, Problem \eqref{eq:consensusprob} has a non-empty and bounded optimal solution set.
\end{assumption}
\begin{assumption}\label{asm:strongconvexity}
    Each $f_i$ is $\mu_i$-strongly convex.
\end{assumption}

 
 We are now in a position to quantify the gap between the fixed point of DGD and the optimal solution. 
%
%
%
Define
\begin{align}
    &L=\max_{i\in\mc{V}} L_i,~\bar{L}=\frac{1}{n}\sum_{i\in\mc{V}} L_i,\label{eq:LbarLdef}\\
    &\beta=\max\{|\lambda_2(W)|, |\lambda_n(W)|\},\label{eq:beta}
\end{align}
where $\beta\in(0,1)$ since $\mc{G}$ is connected \cite{nedic17}.  We first state the following lemma that follows similarly to Lemma 2 and Theorem 4 in \cite{yuan2016convergence}. 
\begin{lemma}\label{lemma:DGDoptimalitygap}
    Suppose that Assumption \ref{asm:prob} holds. If
    \begin{equation*}
        \alpha \le \min\left(\frac{1+\lambda_n(W)}{L}, \frac{1}{\bar{L}}\right),
    \end{equation*}
    then the fixed point set of DGD \eqref{eq:DGD} is non-empty, and DGD converges to a point $\bx^\star\in\mbb{R}^{nd}$ satisfying
    \begin{align}
        & \|x_i^\star - \bar{x}^\star\|\le \frac{\alpha\sqrt{C}}{1-\beta},\quad\forall i\in\mc{V},\label{eq:consensusDGD}\\
        & f(\bar{x}^\star) - f^\star \le \left(\frac{\alpha}{1-\beta}+\frac{L\alpha^2}{2(1-\beta)^2}\right)C,\label{eq:fbarDGD}
    \end{align}
    where $\bar{x}^\star=\frac{1}{n}\sum_{i\in\mc{V}} x_i^\star$, $f^\star$ is the optimal value of \eqref{eq:consensusprob}, $L$, $\bar{L}$, and $\beta$ are given in \eqref{eq:LbarLdef}--\eqref{eq:beta}, and
    \begin{equation}\label{eq:Cdef}
        C=2L(f^\star-\sum_{i\in\mc{V}} \inf_{x_i\in\mbb{R}^d} f_i(x_i)).
    \end{equation}
    The fixed point is unique if, in addition, Assumption \ref{asm:strongconvexity} holds.
\end{lemma}
\begin{proof}
    See Appendix \ref{append:gapDGD}.
\end{proof}

\subsection{DGD using Adapt-Then-Combine Technique (DGD-ATC)}

DGD-ATC \cite{pu2020asymptotic} is a variant of DGD that uses the adapt-then-combine technique and follows the update
%
\begin{equation}\label{eq:DGD-ATC}
    \mathbf{x}^{k+1}= \mb{W}(\bx^k - \alpha \nabla F(\bx^k)),
\end{equation}
where $\mb{W}$ is the same as in \eqref{eq:DGD} and $\alpha>0$ is the step-size.

We are unaware of any previous work that analyses the convergence of DGD-ATC with fixed step-sizes. In the lemma below, we show that DGD-ATC has a similar optimality gap as DGD. The convergence of DGD-ATC \eqref{eq:DGD-ATC} follows as a special case of Theorem \ref{thm:partial} in Section \ref{sec:asycalg}.
\begin{lemma}\label{lemma:DGDATCoptimalitygap}
    Suppose that Assumption \ref{asm:prob} holds. If $W\succ \mb{0}$, then the fixed point set of DGD-ATC \eqref{eq:DGD-ATC} is non-empty and for any fixed point $\bx^\star\in\mbb{R}^{nd}$, \eqref{eq:consensusDGD}--\eqref{eq:fbarDGD} hold. If, in addition, Assumption \ref{asm:strongconvexity} holds, then the fixed point is unique.
\end{lemma}
\begin{proof}
See Appendix \ref{append:DGDATCoptimalitygap}.
\end{proof}
\section{Asynchronous distributed Algorithms}\label{sec:asycalg}

In this section, we introduce the asynchronous DGD and DGD-ATC algorithms. A key advantage of these algorithms is that they do not require global synchronization between nodes or a global clock. 
Both algorithms are analyzed in a setting where each node $i\in\mc{V}$ is activated at discrete time points, and can update and share its local variables once it is activated. In addition, every node $i\in\mc{V}$ has a buffer $\mc{B}_i$ in which it can receive and store messages from neighbors all the time (even when it is inactive).

\subsection{Asynchronous DGD}\label{ssec:algDGD}

In the asynchronous DGD, we let each node $i\in\mc{V}$ hold $x_i\in\mathbb{R}^d$ and $x_{ij}\in\mathbb{R}^d$ $\forall j\in\mc{N}_i$, where $x_i$ is the current local iterate of node $i$ and $x_{ij}$ records the most recent $x_j$ it received from node $j\in\mc{N}_i$. Once activated, node $i$ reads all $x_j$ in the buffer $\mc{B}_i$ and then sets $x_{ij}=x_j$ and, in case $\mc{B}_i$ contains multiple $x_j$'s for a particular $j\in\mc{N}_i$, node $i$ sets  $x_{ij}$ as the most recently received $x_j$. Next, it updates $x_i$ by
\begin{equation}\label{eq:DGDxiupdate}
    x_i \leftarrow w_{ii}x_i+\sum_{j\in\mc{N}_i} w_{ij}x_{ij} - \alpha \nabla f_i(x_i)
\end{equation}
and broadcasts the new $x_i$ to all its neighbors. Once a node $j\in\mc{N}_i$ receives $x_i$, it stores $x_i$ in its buffer $\mc{B}_j$. A detailed implementation is given in Algorithm \ref{alg:DGD}.

To describe the asynchronous DGD mathematically, we index the iterates by $k\in\N_0$. The index $k$ is increased by $1$ whenever an update is performed on a local variable $x_i$ of some nodes $i\in\mc{V}$. The index $k$ does not need to be known by the nodes -- it is only introduced to order events in our theoretical analysis. We can now see that each $x_{ij}$ in \eqref{eq:DGDxiupdate} is a delayed $x_j$ -- each node $i\in\mc{V}$ updates using the most recently received $x_j$ for higher efficiency but it is, in general, not the newest $x_j$ computed by node $j$. Let $\mc{K}_i\subseteq \N_0$ denote the set of iterations where node $i$ updates its iterate. For convenient notation, we define $\bar{\mc{N}}_i=\mc{N}_i\cup\{i\}$ for all $i\in\mc{V}$. Then, the asynchronous DGD can be described as follows. For each $i\in\mc{V}$ and $k\in\N_0$,
\begin{equation}\label{eq:asyDGDupdateindex}
    x_i^{k+1} = \begin{cases}
        \sum_{j\in\bar{\mc{N}_i}} w_{ij}x_j^{s_{ij}^k} - \alpha \nabla f_i(x_i^k), & k\in \mc{K}_i,\\
        x_i^k, & \text{otherwise},
    \end{cases}
\end{equation}
where $s_{ij}^k\in [0, k]$ for $j\in\mc{N}_i$ is the iteration index of the most recent version of $x_j$ available to node $i$ at iteration $k$ and $s_{ii}^k=k$. If $\mc{K}_i=\N_0$ $\forall i\in\mc{V}$ and $s_{ij}^k=k$ $\forall \{i,j\}\in\mc{E}, \forall k\in\N_0$, then \eqref{eq:asyDGDupdateindex} reduces to the synchronous DGD \eqref{eq:DGD}.
\begin{algorithm}[tb]
		\caption{Asynchronous DGD}
		\label{alg:DGD}
		\begin{algorithmic}[1]
			\STATE {\bfseries Initialization:} All the nodes agree on $\alpha>0$, and cooperatively set $w_{ij}$ $\forall \{i,j\}\in\mc{E}$.
			\STATE Each node $i\in\mc{V}$ chooses $x_i\in\mathbb{R}^d$, creates a local buffer $\mc{B}_i$, and shares $x_i$ with all neighbors in $\mc{N}_i$.
			\FOR{each node $i\in \mc{V}$}
			    \STATE 
			    keep \emph{receiving $x_j$ from neighbors and store $(x_j,j)$ in $\mc{B}_i$} until node $i$ is activated\footnotemark.
			    \STATE set $x_{ij}=x_j$ $\forall (x_j,j)\in\mc{B}_i$. If multiple $(x_j,j)\in \mc{B}_i$ for some $j$, choose the most recently received one.
			    \STATE empty $\mc{B}_i$.
			    \STATE update $x_i$ according to \eqref{eq:DGDxiupdate}.
			    \STATE send $x_i$ to all neighbors $j\in\mc{N}_i$.
			\ENDFOR
			\STATE \textbf{Until} a termination criterion is met.
		\end{algorithmic}
	\end{algorithm}
    \footnotetext{\label{note1}In the first iteration, each node $i\in\mc{V}$ can be activated only after it received $(x_j,j)$ (Algorithm \ref{alg:DGD}) or $(y_j,j)$ (Algorithm \ref{alg:DGD-ATC}) from all $j\in\mc{N}_i$.}

\subsection{Asynchronous DGD-ATC}
To implement the asynchronous DGD-ATC, each node $i\in\mc{V}$ holds $x_i\in\mathbb{R}^d$, $y_i\in\mathbb{R}^d$, and $y_{ij}\in\mathbb{R}^d$ for $j\in \mc{N}_i$, where $x_i$ is the current local iterate of node $i$, $y_i=x_i - \alpha \nabla f_i(x_i)$, and $y_{ij}$, $j\in\mc{N}_i$ records the most recent value of $y_j$ it received from node $j$. Once activated, node $i\in\mc{V}$ first reads all $y_j$ in its buffer $\mc{B}_i$ and then sets $y_{ij}=y_j$ and, in case $\mc{B}_i$ contains multiple values of $y_j$ for a particular $j\in\mc{N}_i$, node $i$ sets $y_{ij}$ as the most recent $y_j$ it has received. Next, it updates $x_i$ by
\begin{equation}\label{eq:ATCxiupdate}
    x_i \leftarrow w_{ii}y_i+\sum_{j\in \mc{N}_i} w_{ij}y_{ij},
\end{equation}
computes $y_i=x_i-\alpha \nabla f_i(x_i)$, and broadcasts $y_i$ to all $j\in\mc{N}_i$. Once a node $j\in\mc{N}_i$ receives $y_i$, it stores $y_i$ in its buffer $\mc{B}_j$. A detailed implementation of the asynchronous DGD-ATC is described in Algorithm \ref{alg:DGD-ATC}.

	\begin{algorithm}[tb]
		\caption{Asynchronous DGD-ATC}
		\label{alg:DGD-ATC}
		\begin{algorithmic}[1]
			\STATE {\bfseries Initialization:} All the nodes agree on $\alpha>0$, and cooperatively set $w_{ij}$ $\forall \{i,j\}\in\mc{E}$.
			\STATE Each node $i\in\mc{V}$ chooses $x_i\in\mathbb{R}^d$, creates a local buffer $\mc{B}_i$, sets $y_i=x_i-\alpha\nabla f_i(x_i)$ and shares it with all $j\in\mc{N}_i$.
			\FOR{each node $i\in \mc{V}$}
			\STATE keep \emph{receiving $y_j$ from neighbors and store $(y_j,j)$ in $\mc{B}_i$} until node $i$ is activated\cref{note1}.
			\STATE set $y_{ij}=y_j$ for all $(y_j,j)\in\mc{B}_i$. If multiple $(y_j,j)\in \mc{B}_i$ for some $j$, choose the most recently received one.
			\STATE empty $\mc{B}_i$.
			\STATE update $x_i$ by \eqref{eq:ATCxiupdate}.
			\STATE set $y_i=x_i-\alpha\nabla f_i(x_i)$.
            \STATE share $y_i$ with all neighbors $j\in\mc{N}_i$.
            \ENDFOR
			\STATE \textbf{Until} a termination criterion is met.
		\end{algorithmic}
	\end{algorithm}

Note that each $y_{ij}$ in \eqref{eq:ATCxiupdate} is a delayed $x_j-\alpha \nabla f_j(x_j)$. Then, similar to \eqref{eq:asyDGDupdateindex}, the asynchronous DGD-ATC can be described as follows. For each $i\in\mc{V}$ and $k\in\N_0$,
\begin{equation}\label{eq:ATCupdateindex}
    x_i^{k+1} \!\!=\! \begin{cases}
        \sum_{j\in\bar{\mc{N}_i}} w_{ij}(x_j^{s_{ij}^k}\!-\!\alpha\nabla f_j(x_j^{s_{ij}^k})), & k\in \mc{K}_i,\\
        x_i^k, &\! \text{otherwise},
    \end{cases}
\end{equation}
where all the notations follow their definitions in Section \ref{ssec:algDGD}. When $\mc{K}_i=\N_0$ $\forall i\in\mc{V}$ and $s_{ij}^k=k$ $\forall \{i,j\}\in\mc{E}, \forall k\in\N_0$, \eqref{eq:ATCupdateindex} reduces to the synchronous DGD-ATC.

\section{Convergence Analysis}\label{sec:convana}
In this section, we analyse the convergence of the asynchronous DGD and the asynchronous DGD-ATC under two different models of asynchrony. Our first results allow for total asynchrony in the sense of Bertsekas and Tsitsiklis~\cite{bertsekas2015parallel}, \emph{i.e.} the information delays $k-s_{ij}^k$ may grow arbitrarily large but no node can cease to update and old information must eventually be purged from the system. This assumption is well-suited for scenarios where communication and computation delays are ``unstable", e.g., in massively parallel computing grids with heterogeneous computing nodes, delays can quickly add up if a node is saturated \cite{pmlr-v80-zhou18b}. More formally, we make the following assumption. 

\begin{assumption}[total asynchrony]\label{asm:totalasynchrony}
    The following holds:
	\begin{enumerate}
	    \item $\mc{K}_i$ is an infinite subset of $\N_0$ for each $i\in\mc{V}$.
	    \item $\lim_{k\rightarrow+\infty} s_{ij}^k= +\infty$ for any $i\in\mc{V}$ and $j\in\mc{N}_i$.
	\end{enumerate}
\end{assumption}
%

The following theorem provides delay-free step-size conditions that guarantee that the asynchronous DGD and DGD-ATC algorithms converge under total asynchrony.

\begin{theorem}[total asynchrony]\label{thm:total}
    Suppose that Assumptions \ref{asm:prob}--\ref{asm:totalasynchrony} hold. Also suppose that in the asynchronous DGD,
    \begin{equation}\label{eq:stepsizecond}
        \alpha\in \left(0, 2\min_{i\in\mc{V}} \frac{w_{ii}}{L_i}\right),
    \end{equation}
    and in the asynchronous DGD-ATC,
    \begin{equation}\label{eq:stepsizecondATC}
        \alpha\in\left(0,\frac{2}{\max_{i\in\mc{V}} L_i}\right).
    \end{equation}
    Then, $\{\bx^k\}$ generated by either method converges to some element in the fixed point set of the synchronous counterpart.
\end{theorem}
\begin{proof}
See Appendix \ref{append:thmtotal}.
\end{proof}

Under total asynchrony, 
there is no 
lower bound on the update frequency of nodes and no upper bound on the information delays, and we are only able to give asymptotic convergence guarantees.
To derive non-asymptotic convergence rate guarantees, we consider the more restrictive notion of partial asynchrony~\cite{bertsekas2015parallel}.
\begin{assumption}[partial asynchrony]\label{asm:partialasynchrony}
    There exist positive integers $B$ and $D$ such that
    \begin{enumerate}
	\item For every $i\in\mc{V}$ and for every $k\ge 0$, at least one element in the set $\{k,\ldots,k+B\}$ belongs to $\mc{K}_i$.
	  \item There holds
        \begin{equation*}
            k-D \le s_{ij}^k \le k
        \end{equation*}
        for all $i\in\mc{V}$, $j\in\mc{N}_i$, and $k\in \mc{K}_i$.
    \end{enumerate}
\end{assumption}
In Assumption \ref{asm:partialasynchrony}, $B$ and $D$ characterize the minimum update frequency and the maximal information delay, respectively. If $B=D=0$, then Assumption \ref{asm:partialasynchrony} reduces to the synchronous scheme where all local variables $x_i^k$ $\forall i\in\mc{V}$ are instantaneously updated at every iteration $k\in\N_0$.

To state our convergence result, we define the block-wise maximum norm for any $\bx=(x_1^T,\ldots, x_n^T)^T\in\mathbb{R}^{nd}$ as
\begin{equation*}
    \|\bx\|_{\infty}^b = \max_{i\in\mc{V}} \|x_i\|.
\end{equation*}
The following theorem establishes linear convergence for the two algorithms under partial asynchrony.
\begin{theorem}[partial asynchrony]\label{thm:partial}
    Suppose that Assumptions \ref{asm:prob}, \ref{asm:strongconvexity}, \ref{asm:partialasynchrony} hold. Also suppose that \eqref{eq:stepsizecond} holds in the asynchronous DGD and \eqref{eq:stepsizecondATC} holds in the asynchronous DGD-ATC. Then, $\{\bx^k\}$ generated by either method satisfies
    \begin{equation*}
        \|\mathbf{x}^k-\mathbf{x}^\star\|_{\infty}^b\le \rho^{\lfloor k/(B+D+1)\rfloor}\|\mathbf{x}^0-\mathbf{x}^\star\|_{\infty}^b,
    \end{equation*}
    where $\bx^\star$ is the fixed point of their synchronous counterpart and $\rho\in(0,1)$. Specifically, for
    \begin{align}
        &\text{async DGD}:~\rho=\sqrt{1-\alpha\min_{i\in\mc{V}}\left(\mu_i\left(2-\frac{\alpha L_i}{w_{ii}}\right)\right)},\label{eq:rhodgd}\\
        &\text{async DGD-ATC}:~\rho=\sqrt{1-\alpha\min_{i\in\mc{V}}\left(\mu_i(2-\alpha L_i)\right)}.\label{eq:rhodgdatc}
    \end{align}
\end{theorem}
\begin{proof}
See Appendix \ref{append:thmpartial}.
\end{proof}

By Lemmas \ref{lemma:DGDoptimalitygap}--\ref{lemma:DGDATCoptimalitygap} and Theorems \ref{thm:total}-\ref{thm:partial}, the two asynchronous methods can converge to an approximate optimum of Problem \eqref{eq:consensusprob}, where the optimality gap is given in Lemmas \ref{lemma:DGDoptimalitygap}--\ref{lemma:DGDATCoptimalitygap}. Note that the range of step-sizes that guarantees convergence is independent of the degree of asynchrony in the system. The two algorithms 
converge even under total asynchrony, but the guarantees we can give improve as the amount of asynchrony decreases. Moreover, Theorem \ref{thm:partial} indicates two advantages of the asynchronous DGD-ATC over the asynchronous DGD. First, it allows for a larger step-size range \eqref{eq:stepsizecondATC} than \eqref{eq:stepsizecond}, which may lead to faster practical convergence. Second, even using the same $\alpha$, the asynchronous DGD-ATC has a faster convergence rate: Let $\rho$, $\hat{\rho}$ denote the values in \eqref{eq:rhodgd} and \eqref{eq:rhodgdatc}, respectively. Then,
         \begin{equation}\label{eq:fastrate}
             \begin{split}
                 \rho^2-\hat{\rho}^2 &= \alpha \left(\min_{i\in\mc{V}}\mu_i(2-\alpha L_i)-\min_{i\in\mc{V}} \mu_i\left(2-\alpha \frac{L_i}{w_{ii}}\right)\right)\\
                 &\ge \alpha\min_{i\in\mc{V}} \left(\mu_i(2-\alpha L_i)-\mu_i\left(2-\alpha \frac{L_i}{w_{ii}}\right)\right)\\
                 &= \alpha^2\min_{i\in\mc{V}} \mu_iL_i\left(\frac{1}{w_{ii}}-1\right)\ge 0.
             \end{split}
         \end{equation}
        The faster convergence of the asynchronous DGD-ATC is also demonstrated by experiments in Section \ref{sec:exp}.
\subsection{Comparison with Related Methods}

To the best of our knowledge, Theorem \ref{thm:total} provides the first convergence result for solving \eqref{eq:consensusprob} with non-quadratic $f_i$ on general networks under total asynchrony. Other works considering total asynchrony include \cite{ubl2021totally, mishchenko2018delay}. In particular, the asynchronous coordinate descent method in \cite{ubl2021totally} can solve Problem \eqref{eq:consensusprob} with quadratic objective functions over undirected, connected networks, and the asynchronous proximal gradient method in \cite{mishchenko2018delay} can address \eqref{eq:consensusprob} with non-quadratic $f_i$'s, but only considers star networks.

In order to distinguish our results from the state-of-the-art on asynchronous consensus optimization algorithms \cite{nedic2010convergence,zhang2019asyspa,sirb2016consensus,doan2017convergence,assran2020asynchronous,zhang2019fully,wu2017decentralized,tian2020achieving, kungurtsev2023decentralized,ubl2021totally,mishchenko2018delay}, we categorize these works based on their step-sizes and compare them to our results.

\noindent\textbf{\emph{delay-dependent step-size}}:\cite{wu2017decentralized, tian2020achieving,zhang2019fully} assume the existence of an upper bound on the information delay and use fixed parameters relying on and decreasing with the delay bound. Although the works \cite{wu2017decentralized, tian2020achieving,zhang2019fully} can achieve convergence to the exact optimum under partial asynchrony, which is more desirable than the inexact convergence of our algorithms, they suffer from difficult parameter determination and unnecessary slow convergence for two reasons. Firstly, the delay bound is often unknown and hard to obtain in advance. Secondly, the delay bound is typically large, which leads to small step-sizes and further slows down the convergence process. Our numerical experiments in Section \ref{sec:exp} suggest that the asynchronous DGD and DGD-ATC can significantly outperform PG-EXTRA \cite{wu2017decentralized} for the simulated problem. In addition, our algorithms can converge under total asynchrony that is not allowed in \cite{wu2017decentralized, tian2020achieving,zhang2019fully}.

\noindent \textbf{\emph{delay-free and non-diminishing step-size}}: This category includes \cite{nedic2010convergence, ubl2021totally, mishchenko2018delay}. However, \cite{nedic2010convergence, ubl2021totally} can only solve simple problems. The work \cite{nedic2010convergence} focuses on the consensus problem which is equivalent to Problem \eqref{eq:consensusprob} with $f_i(x)\equiv 0$, and \cite{ubl2021totally} can only deal with Problem \eqref{eq:consensusprob} with quadratic objective functions. The work \cite{mishchenko2018delay} can solve Problem \eqref{eq:consensusprob} with non-quadratic objective functions, but requires star networks. In contrast, our results in Theorem \ref{thm:total}--\ref{thm:partial} allow for non-quadratic objective functions and non-star communication networks, which is a substantial improvement.

\noindent\textbf{\emph{diminishing step-size}}: \cite{zhang2019asyspa,sirb2016consensus,doan2017convergence,assran2020asynchronous,kungurtsev2023decentralized} consider diminishing step-sizes that are also delay-free. However, the diminishing step-sizes decrease rapidly and can lead to slow practical convergence. Moreover, \cite{zhang2019asyspa,sirb2016consensus,doan2017convergence,assran2020asynchronous,kungurtsev2023decentralized} all focus on partial asynchrony, while our algorithms can converge under total asynchrony.

\section{Numerical Experiments}\label{sec:exp}

We evaluate the practical performance of the asynchronous DGD and the asynchronous DGD-ATC on decentralized learning using the $\ell_2$- regularized logistic loss:
\begin{equation}\label{eq:simuprob}
    \begin{split}
        \underset{x\in\mathbb{R}^d}{\operatorname{minimize}}~\frac{1}{N}\sum_{i=1}^N \left(\log(1+e^{-b_i(a_i^Tx)})+\frac{\lambda}{2}\|x\|^2\right),
    \end{split}
\end{equation}
where $N$ is the number of samples, $a_i$ is the feature of the $i$th sample, $b_i$ is the corresponding label, and $\lambda=10^{-3}$ is the regularization parameter. The experiments use the training set of Covertype \cite{Dua:2019} and MNIST \cite{lecun1998gradient} summarized below:
\begin{table}[!htb]
		\centering
		\caption{Information about training data sets.}
		\begin{tabular}{c|c|c}
			\hline
			Data set & sample number $N$ & feature dimension $d$\\
			\hline
			Covertype & 581012 & 54\\
			\hline
			MNIST & 60000 & 784\\
			\hline
		\end{tabular}
		\label{tab:dataset}
	\end{table}

We compare our algorithms with the asynchronous PG-EXTRA \cite{wu2017decentralized}. We do not compare with the algorithms in~\cite{zhang2019asyspa,sirb2016consensus,doan2017convergence,kungurtsev2023decentralized,assran2020asynchronous} with diminishing step-sizes because \cite{zhang2019asyspa,sirb2016consensus,doan2017convergence,kungurtsev2023decentralized} require Lipschitz continuous objective functions which does not hold for Problem \eqref{eq:simuprob} and the maximum allowable step-size in \cite{assran2020asynchronous} is excessively small ($\le 10^{-10}$ in our experiment setting). We set $n=8$, evenly partition and allocate all data samples to each node, and implement all the methods on a multi-core computer using the message-passing framework MPI4py \cite{dalcin2008mpi}, where each core serves as a node and the communication graph $\mc{G}$ is displayed in Figure \ref{fig:communicationgraph}. 
In the experiments, each node $i\in\mc{V}$ is activated once its buffer $\mc{B}_i$ is non-empty, and the delays are generated by real interactions between the nodes and not by any theoretical delay model. We set $\alpha=\min_{i\in\mc{V}} w_{ii}/\max_{i\in\mc{V}} L_i$ in the asynchronous DGD and $\alpha=1/\max_{i\in\mc{V}} L_i$ in the asynchronous DGD-ATC, which meet their conditions in Theorems \ref{thm:total}--\ref{thm:partial}. We fine-tune the parameters of the asynchronous PG-EXTRA within their theoretical ranges for guaranteeing convergence. The theoretical ranges involve the maximum delay, which is determined by recording the maximum observed delay during a 20-second run of the method.

\begin{figure}[!htb]
    \centering    \includegraphics[scale=0.45]{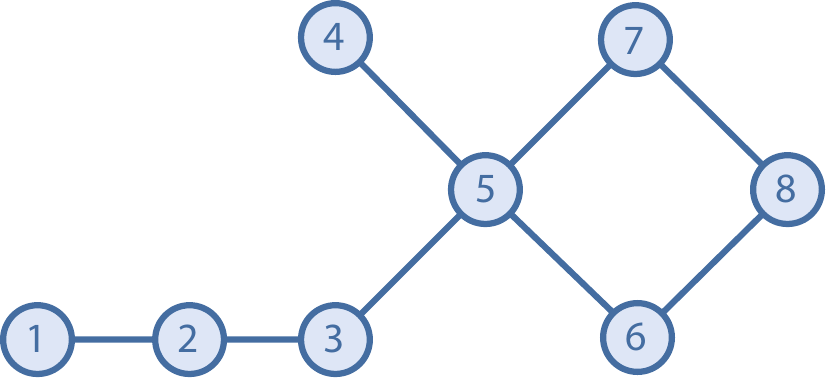}
    \vspace{0.4cm}
    
    \caption{Communication graph in simulation.}
    \label{fig:communicationgraph}
\end{figure}

\begin{figure}[!htb]
	\centering
	\subfigure[Covertype]{
		\includegraphics[scale=0.77]{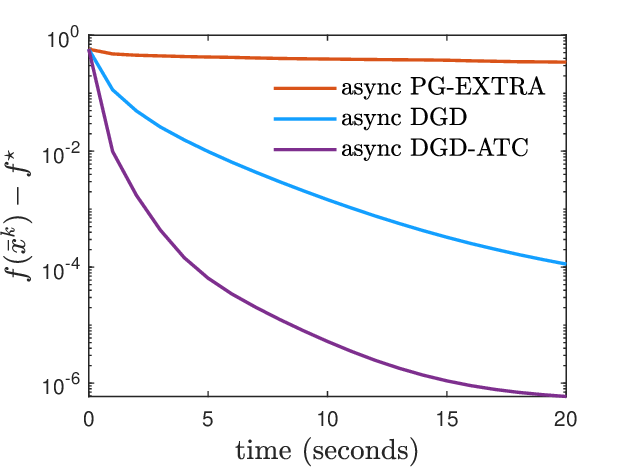}\label{fig:covtype}}\\
	\subfigure[MNIST]{
		\includegraphics[scale=0.77]{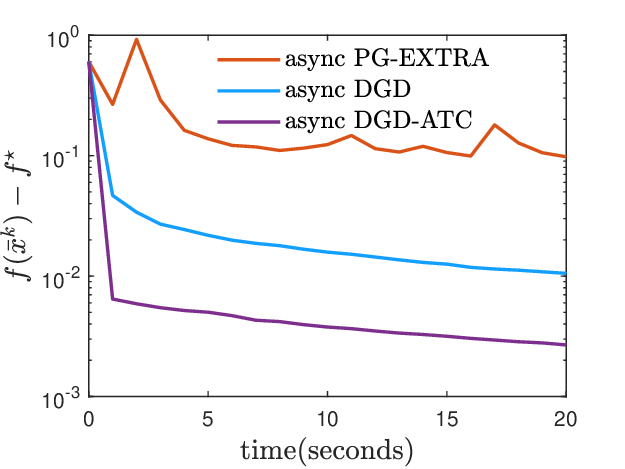}\label{fig:mnist}}
	\caption{Convergence on logistic regression}
	\label{fig:conv}
\end{figure}

We run all methods for $20$ seconds and plot the training error $f(\bar{x}^k)-f^\star$ at the average iterate $\bar{x}^k=\frac{1}{n}\sum_{i=1}^n x_i^k$ in Figure \ref{fig:conv}, where $f^\star$ is the optimal value of \eqref{eq:simuprob}. We can see that for both datasets, the asynchronous DGD-ATC outperforms the asynchronous DGD as indicated by \eqref{eq:fastrate}, and they both converge faster than the asynchronous PG-EXTRA. The slow convergence of the asynchronous PG-EXTRA may be because of its conservative parameters caused by the large delay, while our algorithms can converge under much more relaxed delay-free parameter conditions.

\section{Conclusion}

We have investigated the asynchronous version of two distributed algorithms, DGD and DGD-ATC, for solving consensus optimization problems. We first reviewed existing results on the optimality gap
of DGD and developed a corresponding results for the optimality gap of DGD-ATC.
Then, we developed \emph{delay-free} parameter conditions under which both asynchronous methods converge to the fixed point set of their synchronous counterparts under total and partial asynchrony. Finally, we demonstrated superior practical convergence of the two asynchronous algorithms via numerical experiments. Future work includes developing asynchronous algorithms with delay-free parameter conditions for other distributed optimization problems.

\appendix
\subsection{Proof of Lemma \ref{lemma:DGDoptimalitygap}}\label{append:gapDGD}
The results in [17] implicitly assume that there exists a fixed point to DGD. However, this is not straightforward in general. Thus, we include a proof to show the existence of the fixed point.
\subsubsection{Non-empty fixed point set of \eqref{eq:consensusprob} under Assumption \ref{asm:prob}} Let $z^\star$ be an optimum to \eqref{eq:consensusprob} and $\bz^\star=\mb{1}_n\otimes z^\star$. Define $L_{\alpha}(\bx) = F(\bx)+\frac{\|\bx\|_{I-\mb{w}}^2}{2\alpha}$. It can be verified that every minimum of $L_{\alpha}$ is a fixed point of \eqref{eq:DGD}. Therefore, to show the fixed point set of \eqref{eq:DGD} is non-empty, it suffices to show the minimum of $L_{\alpha}$ exists. Define 
\begin{equation*}
    \mc{S} = \{\bx: L_{\alpha}(\bx)\le L_{\alpha}(\bz^\star)\}.
\end{equation*}
Since $\min_{\bx\in\mbb{R}^{nd}} L_{\alpha}(\bx)$ is equivalent to $\min_{\bx\in\mc{S}} L_{\alpha}(\bx)$, the minimum of $L_{\alpha}$ exists if the optimum of the later problem exists which can be guaranteed by the nonemptiness and compactness of $\mc{S}$. Clearly, $\mc{S}$ is non-empty since $\bz^\star\in \mc{S}$.

Below, we prove that $\mc{S}$ is compact. To this end, fix $\bx\in \mc{S}$ and define $h_i=\inf_{y\in\mbb{R}^d} f_i(y)$ $\forall i\in\mc{V}$ and $h=\sum_{i\in\mc{V}} h_i$. By the Lipschitz continuity of $\nabla f_i$,
    \begin{equation}\label{eq:hismallerthan}
        h_i\le f_i(x_i-\frac{1}{L_i}\nabla f_i(x_i))\le f_i(x_i) - \frac{1}{2L_i}\|\nabla f_i(x_i)\|^2.
    \end{equation}
    Because $\bx\in \mc{S}$,
    \begin{equation}\label{eq:Fsmallerthanfstar}
        F(\bx)\le L_{\alpha}(\bx)\le L_{\alpha}(\bz^\star)=f^\star,
    \end{equation}
    which, together with \eqref{eq:hismallerthan}, yields
    \begin{equation}\label{eq:boundedgra}
        \|\nabla F(\bx)\|^2 \le 2L(F(\bx) - h)\le 2L(f^\star-h).
    \end{equation}
    Because $F(\bx)\ge h$ and $L_{\alpha}(\bx)\le f^\star$ by \eqref{eq:Fsmallerthanfstar}, we have
    \begin{equation}\label{eq:boundediff}
        \|\bx\|_{I-\mb{W}}^2 = 2\alpha(L_{\alpha}(\bx)-F(\bx))\le 2\alpha(f^\star-h).
    \end{equation}
    Let $\bar{\bx} = \mb{1}_n\otimes\frac{1}{n}\sum_{i\in\mc{V}} x_i$. Since $\mc{G}$ is connected, we have $\operatorname{Range}(I-\mb{W})=\{\by:y_1+\ldots+y_n=0\}$ and $\bar{\bx}-\bx\in \operatorname{Range}(I-\mb{W})$. This, together with \eqref{eq:boundediff} and $\mb{W}\preceq I$, yields
    \begin{equation}\label{eq:gapxbarx}
        \|\bx-\bar{\bx}\|^2\le \frac{\|\bx\|_{I-\mb{W}}^2}{\lambda_{\min}(I-\mb{W})}\le \frac{2\alpha(f^\star-h)}{\lambda_{\min}(I-\mb{W})},
    \end{equation}
    where $\lambda_{\min}(\cdot)$ represents the minimal positive eigenvalue. By the $L$-smoothness of $f$,
    \begin{equation}\label{eq:smoothness}
        \begin{split}
            F(\bar{\bx}) - F(\bx) &\le \langle \nabla F(\bx), \bar{\bx}-\bx\rangle+\frac{L}{2}\|\bar{\bx}-\bx\|^2\\
            &\le \frac{\|\nabla F(\bx)\|\cdot \|\bar{\bx}-\bx\|}{2}+\frac{L}{2}\|\bar{\bx}-\bx\|^2.
        \end{split}
    \end{equation}
    Substituting $F(\bx)\le f^\star$, \eqref{eq:boundedgra}, and \eqref{eq:gapxbarx} into \eqref{eq:smoothness}, we have
    \begin{equation*}
        F(\bar{\bx})\le C_0,
    \end{equation*}
    where $C_0=f^\star+\left(\sqrt{\frac{\alpha L}{\lambda_{\min}(I-\mb{W})}}+\frac{\alpha L}{\lambda_{\min}(I-\mb{W})}\right)(f^\star-h)$. In addition, by \cite[Proposition B.9]{bertsekas1995nonlinear} and the bounded optimum set of Problem \eqref{eq:consensusprob}, we have that every level set of $f$ is bounded, which yields the compactness of
    \begin{align}\label{eq:levelset}
        \{y\in\mbb{R}^d: f(y)\le C_0\}.
    \end{align}
    Due to the arbitrariness of $\bx\in \mc{S}$, we have that for any $\bx\in\mc{S}$, $\frac{1}{n}\sum_{i=1}^n x_i$ belongs to the compact set \eqref{eq:levelset} and \eqref{eq:gapxbarx} holds. Therefore, $\mc{S}$ is compact. Concluding all the above, the fixed point set of DGD \eqref{eq:DGD} is non-empty.

\subsubsection{Optimality gap and uniqueness of fixed point}
The consensus error bound \eqref{eq:consensusDGD} can be directly obtained by letting $x_i^k=x_i^\star$ in \cite[Lemma 2]{yuan2016convergence}. The convergence of DGD follows that of gradient descent because DGD is equivalent to gradient descent for minimizing $L_{\alpha}$.

Next, we prove \eqref{eq:fbarDGD}. Because $\bx^\star$ is a fixed point of \eqref{eq:DGD},
\begin{equation}\label{eq:fpdgd}
    (I-\mb{W})\bx^\star=-\alpha\nabla F(\bx^\star).
\end{equation}
Let $\bar{\bx}^\star=\mb{1}_n\otimes \bar{x}^\star$. Because $\mb{1}_n$ is the eigenvector of $W$ corresponding to the unique maximal eigenvalue $1$, we have
\begin{equation}
    \|W-\frac{\mb{1}_n\mb{1}_n^T}{n}\|\le \max (|\lambda_2(W)|, |\lambda_n(W)|)=\beta.
\end{equation}
Then, 
\begin{equation}\label{eq:Wxstar}
	\begin{split}
	\|\mb{W}(\bx^\star-\bar{\bx}^\star)\| &= \|(\mb{W}-\frac{\mb{1}_n\mb{1}_n^T}{n}\otimes I_d)(\bx^\star-\bar{\bx}^\star)\|\\
	&\le \|\mb{W}-\frac{\mb{1}_n\mb{1}_n^T}{n}\otimes I_d\|\cdot\|\bx^\star-\bar{\bx}^\star\|\\
	&\le \beta\|\bx^\star-\bar{\bx}^\star\|.
	\end{split}
	\end{equation}
	By \eqref{eq:Wxstar},
	\begin{equation*}
	\begin{split}
	\|(I-\mb{W})\bx^\star\| &= \|(I-\mb{W})(\bx^\star-\bar{\bx}^\star)\|\\
	&\ge \|\bx^\star-\bar{\bx}^\star\|-\|\mb{W}(\bx^\star-\bar{\bx}^\star)\|\\
	&\ge (1-\beta)\|\bx^\star-\bar{\bx}^\star\|,
	\end{split}
	\end{equation*}
 which, together with \eqref{eq:fpdgd}, yields
 \begin{equation}\label{eq:conserrorboundbygra}
	\|\bx^\star-\bar{\bx}^\star\|\le \frac{\alpha}{1-\beta}\|\nabla F(\bx^\star)\|.
\end{equation}
By letting $\bx=\bx^\star$ in \eqref{eq:boundedgra} and \eqref{eq:smoothness} and using \eqref{eq:conserrorboundbygra},
\begin{equation}\label{eq:fbarandnonbar}
	\begin{split}
	F(\bar{\bx}^\star) - F(\bx^\star) &\le \|\nabla F(\bx^\star)\|\cdot\|\bar{\bx}^\star-\bx^\star\|+\frac{L}{2}\|\bar{\bx}^\star-\bx^\star\|^2\\
	&\le \left(\frac{\alpha}{1-\beta}+\frac{L\alpha^2}{2(1-\beta)^2}\right)C.
	\end{split}
	\end{equation}
 Also, $F(\bx^\star)\le L_{\alpha}(\bx^\star)\le L_{\alpha}(\bz^\star)=f^\star$. Then we have \eqref{eq:fbarDGD}.
 
If all the $f_i$'s are strongly convex, the function $L_{\alpha}$ is strongly convex and, therefore, its minimum is unique. Note that every fixed point of DGD \eqref{eq:DGD} is a minimum of $L_{\alpha}$ and vice versa. Thus, the fixed point of DGD \eqref{eq:DGD} is also unique.

\subsection{Proof of Lemma \ref{lemma:DGDATCoptimalitygap}}\label{append:DGDATCoptimalitygap}
Define $\tilde{L}_{\alpha}(\bx) = F(\bx)+\frac{\|\bx\|_{\mb{W}^{-1}-I}^2}{2\alpha}$. Note that we assume $W$ is invertible, so is $\mb{W}$. Moreover, every minimum of $\tilde{L}_{\alpha}$ is a fixed point of \eqref{eq:DGD-ATC} and vice versa. By almost the same proof with that of Lemma \ref{lemma:DGDoptimalitygap}, the minimum of $\tilde{L}_{\alpha}$ exists, and it is unique if, in addition, each $f_i$ is strongly convex. Therefore, the fixed point set of \eqref{eq:DGD-ATC} is non-empty, and if each $f_i$ is strongly convex, then it is a singleton.

Next, we prove \eqref{eq:consensusDGD}--\eqref{eq:fbarDGD}. Suppose $\bx^\star$ is a fixed point of \eqref{eq:DGD-ATC} and $z^\star$ is an optimum of \eqref{eq:consensusprob}. Let $h$ be a lower bound of $F(\bx)$, which exists due to Assumption \ref{asm:prob}. Similar to \eqref{eq:boundedgra},
\begin{equation}\label{eq:grafstarbound}
    \|\nabla F(\bx^\star)\|^2 \le 2L(f^\star - h)=C.
\end{equation}
Then, by $\mb{x}^\star = \mb{W}(\bx^\star - \alpha\nabla F(\bx^\star))$, we have \begin{equation}\label{eq:proofDGDWinverseminiusI}
    \|(\mb{W}^{-1}-I)\bx^\star\| = \alpha \|\nabla F(\bx^\star)\|\le \alpha\sqrt{C}.
\end{equation}
Let $\bar{\bx}^\star=\mb{1}_n\otimes \bar{x}^\star$. Because $\bar{\bx}^\star-\bx^\star\in \operatorname{Range}(\mb{W}^{-1}-I)=\{\by:y_1+\ldots+y_n=0\}$ and $\lambda_{\min}(\mb{W}^{-1}-I)=\frac{1}{\lambda_2(W)}-1=\frac{1}{\beta}-1$, we have
\begin{equation}\label{eq:consensusstarbound}
    \begin{split}
        &\|x_i^\star-\bar{x}^\star\|\le \|\bx^\star-\bar{\bx}^\star\|\\
        \le&\frac{\beta\|(\mb{W}^{-1}\!-I)\bx^\star\|}{1-\beta}\le \frac{\alpha\sqrt{C}}{1-\beta},
    \end{split}
\end{equation}
where the last step uses \eqref{eq:proofDGDWinverseminiusI}. Therefore, \eqref{eq:consensusDGD} holds.

Let $\bz^\star=\mb{1}_n\otimes z^\star$. By \eqref{eq:grafstarbound} and \eqref{eq:consensusstarbound}, equation \eqref{eq:fbarandnonbar} also holds for DGD-ATC. In addition,
\begin{equation}\label{eq:gapimportantATC}
    F(\bx^\star)\le \tilde{L}_{\alpha}(\bx^\star) \le \tilde{L}_{\alpha}(\bz^\star)=f^\star.
\end{equation}
Substituting \eqref{eq:gapimportantATC} into \eqref{eq:fbarandnonbar} yields \eqref{eq:fbarDGD}. Completes the proof.

\subsection{Proof of Theorem \ref{thm:total}}\label{append:thmtotal}
The proof includes two steps. Step 1 rewrites the two methods as a unified form and introduce a convergence theorem for the unified algorithm form. Step 2 proves that the two asynchronous methods satisfy the conditions in the convergence theorem.

\textbf{Step 1: a unified description for the asynchronous DGD and the asynchronous DGD-ATC.} Both DGD \eqref{eq:DGD} and DGD-ATC \eqref{eq:DGD-ATC} can be described by the general fixed-point update:
\begin{equation}\label{eq:fpu}
    \mathbf{x}^{k+1} = \T(\mathbf{x}^k),
\end{equation}
where $\T:\mathbb{R}^{nd}\rightarrow\mathbb{R}^{nd}$ is a function and
\begin{align}
\textbf{DGD:}~&~\T(\bx)= \mathbf{W}\bx - \alpha \nabla F(\bx),\label{eq:T}\\
\textbf{DGD-ATC:}~&~\T(\bx)= \mathbf{W}(\bx - \alpha \nabla F(\bx)).\label{eq:TATC}
\end{align}
In addition, let $\T_i:\mathbb{R}^{nd}\rightarrow\mathbb{R}^d$ be the $i$th block of $\T$ for any $i\in \mc{V}$ and consider the asynchronous version of \eqref{eq:fpu}:
\begin{equation}\label{eq:asyncop}
    x_i^{k+1} = \begin{cases}
        \T_i(\bz_i^k), & k\in\mc{K}_i,\\
        x_i^k, & \text{otherwise},
    \end{cases}
\end{equation}
where $\bz_i^k = (x_1^{t_{i1}^k}, \ldots, x_n^{t_{in}^k})$ for some non-negative integers $t_{ij}^k$. By letting
\begin{equation}
    t_{ij}^k = 
    \begin{cases}
        s_{ij}^k, & j\in\bar{\mc{N}}_i,\\
        k, & \text{otherwise},
        \end{cases} \forall i\in\mc{V},~k\in\mc{K}_i,
\end{equation}
\eqref{eq:asyncop} with $\T$ in \eqref{eq:T} and \eqref{eq:TATC} describes the asynchronous DGD and the asynchronous DGD-ATC, respectively.

For the asynchronous update \eqref{eq:asyncop}, \cite{Feyzmahdavian21} presents the following convergence results for pseudo-contractive operator $\T$: for some $\rho\in(0,1)$,
\begin{equation}\label{eq:pseudocontractive}
    \!\|\T(\bx)-\bx^\star\|_{\infty}^b\le \rho\|\bx-\bx^\star\|_{\infty}^b,\forall \bx\in\mathbb{R}^{nd}, \bx^\star\in \operatorname{Fix} \T,
\end{equation}
where $\operatorname{Fix} \T$ is the fixed point set of $\T$.
\begin{lemma}[Theorem 3.20, \cite{Feyzmahdavian21}]\label{lemma:total}
Suppose that Assumption \ref{asm:totalasynchrony} holds and $0\in\mc{K}_i$ $\forall i\in\mc{V}$. If \eqref{eq:pseudocontractive} holds for some $\rho\in(0,1)$, then $\{\bx^k\}$ generated by the iteration \eqref{eq:asyncop} converges
asymptotically to the unique fixed point of $\T$.
\end{lemma}
Although Theorem 3.20 in \cite{Feyzmahdavian21} assumes
\begin{align}\label{eq:initialfey}
    0\in \mc{K}_i, \forall i\in\mc{V}
\end{align}
for simplicity of presentation, the convergence still holds without \eqref{eq:initialfey}. With Lemma \ref{lemma:total}, to show Theorem \ref{thm:total}, it suffices to show the pseudo-contractivity \eqref{eq:pseudocontractive} for $\T$ in \eqref{eq:T} and \eqref{eq:TATC}.

\textbf{Step 2: Proof of pseudo-contractivity \eqref{eq:pseudocontractive}}. Let $\rho_{\operatorname{c}}$ and $\rho_{\operatorname{a}}$ be the value in \eqref{eq:rhodgd} and \eqref{eq:rhodgdatc}, respectively. Below, we show \eqref{eq:pseudocontractive} for the two operators in \eqref{eq:T} and \eqref{eq:TATC}.

\textbf{1) $\T$ in \eqref{eq:T}}: For any $i\in\mc{V}$, since $x_i^\star=\T_i(\bx^\star)$,
    \begin{equation}\label{eq:nonexpansiveATC}
    \begin{split}
    &\|\T_i(\bx)-x_i^\star\|^2 = \|\T_i(\bx)-\T_i(\bx^\star)\|^2\\
    =&\|\sum_{j\in\mc{N}_i} w_{ij}(x_j-x_j^\star) +\\
    &\quad w_{ii}(x_i-x_i^\star-\frac{\alpha}{w_{ii}}(\nabla f_i(x_i)-\nabla f_i(x_i^\star)))\|^2\\
    \le&\sum_{j\in\mc{N}_i} w_{ij}\|x_j-x_j^\star\|^2+\\
    & w_{ii}\|x_i-x_i^\star-\frac{\alpha}{w_{ii}}(\nabla f_i(x_i)-\nabla f_i(x_i^\star))\|^2,
    \end{split}
    \end{equation}
    where the last step uses Jensen's inequality on the norm square. Since each $f_i$ is $\mu_i$-strongly convex and $\nabla f_i$ is Lipschitz continuous, by \cite[Equation (2.1.8)]{nesterov2003introductory},
\begin{align}
    &\langle \nabla f_i(x_i)-\nabla f_i(x_i^\star), x_i-x_i^\star\rangle\ge\mu_i\|x_i-x_i^\star\|^2,\label{eq:proofstrongconvexity}\\
    &\langle \nabla f_i(x_i)-\nabla f_i(x_i^\star), x_i-x_i^\star\rangle\ge\frac{\|\nabla f_i(x_i)-\nabla f_i(x_i^\star)\|^2}{L_i}.\label{eq:proofsmooth}
\end{align}
Then,
\begin{align}
&\|x_i-x_i^\star-\frac{\alpha}{w_{ii}}(\nabla f_i(x_i)-\nabla f_i(x_i^\star))\|^2\nonumber\allowdisplaybreaks\\
        =&\|x_i-x_i^\star\|^2-2\frac{\alpha}{w_{ii}}\langle \nabla f_i(x_i)-\nabla f_i(x_i^\star), x_i-x_i^\star\rangle\nonumber\allowdisplaybreaks\\
        &+(\frac{\alpha}{w_{ii}})^2\|\nabla f_i(x_i)-\nabla f_i(x_i^\star)\|^2\allowdisplaybreaks\nonumber\\
        \overset{\eqref{eq:proofsmooth}}{\le}& \|x_i-x_i^\star\|^2\!\!-\!\frac{\alpha}{w_{ii}}(2\!-\!\frac{L_i\alpha}{w_{ii}})\langle \nabla f_i(x_i)\!-\!\nabla f_i(x_i^\star), x_i-x_i^\star\rangle\nonumber\allowdisplaybreaks\\
        \overset{\eqref{eq:proofstrongconvexity}}{\le}&(1-\frac{\alpha}{w_{ii}}(2-\frac{L_i\alpha}{w_{ii}})\mu_i)\|x_i-x_i^\star\|^2\allowdisplaybreaks.\label{eq:shrinkATC}
\end{align}
Substituting \eqref{eq:shrinkATC} into \eqref{eq:nonexpansiveATC} yields
\begin{equation*}
\begin{split}
    \|\T_i(\bx)-x_i^\star\|^2\le \rho_{\operatorname{c}}^2(\|\bx-\bx^\star\|_{\infty}^b)^2,
\end{split}
\end{equation*}
which leads to \eqref{eq:pseudocontractive} with $\rho=\rho_{\operatorname{c}}$.

\textbf{2) $\T$ in \eqref{eq:TATC}}: For any $i\in\mc{V}$, since $x_i^\star=\T_i(\bx^\star)$,
    \begin{equation}\label{eq:nonexpansiveCAA}
        \begin{split}
            &\|\T_i(\bx)-x_i^\star\|^2 = \|\T_i(\bx)-\T_i(\bx^\star)\|^2\\
            =&\|\sum_{j\in\bar{\mc{N}}_i} w_{ij}(x_j-x_j^\star-\alpha(\nabla f_j(x_j)-\nabla f_j(x_j^\star)))\|^2\\
            \le&\sum_{j\in\bar{\mc{N}}_i} w_{ij}\|x_j-x_j^\star-\alpha(\nabla f_j(x_j)-\nabla f_j(x_j^\star))\|^2,
        \end{split}
    \end{equation}
where the last step uses Jensen's inequality on the norm square. Similar to \eqref{eq:shrinkATC} with $\alpha=1$,
\begin{equation*}
    \begin{split}
        &\|x_j-x_j^\star-\alpha(\nabla f_j(x_j)-\nabla f_j(x_j^\star))\|^2\\
        \le&(1-\alpha(2-L_j\alpha))\|x_j-x_j^\star\|^2\\
        \le& \rho_{\operatorname{a}}^2(\|\bx-\bx^\star\|_{\infty}^b)^2.
    \end{split}
\end{equation*}
Substituting the above equation into \eqref{eq:nonexpansiveCAA} yields
\begin{equation*}
        \|\T_i(\bx)-x_i^\star\|^2\le \rho_{\operatorname{a}}^2(\|\bx-\bx^\star\|_{\infty}^b)^2,
\end{equation*}
which results in \eqref{eq:pseudocontractive} with $\rho=\rho_{\operatorname{a}}$ and completes the proof.


\subsection{Proof of Theorem \ref{thm:partial}}\label{append:thmpartial}

The proof uses Theorem 3.21 in \cite{Feyzmahdavian21}.
\begin{lemma}[Theorem 3.21, \cite{Feyzmahdavian21}]\label{lemma:partial}
Suppose that Assumption \ref{asm:partialasynchrony} and \eqref{eq:initialfey} hold. If \eqref{eq:pseudocontractive} holds for some $\rho\in(0,1)$, then $\{\bx^k\}$ generated by the asynchronous iteration \eqref{eq:asyncop} satisfy
\begin{equation}\label{eq:linearfey}
    \|\bx^k-\bx^\star\|_{\infty}^b\le \rho^{\frac{k}{B+D+1}}\|\bx^0-\bx^\star\|_{\infty}^b.
\end{equation}
\end{lemma}
Note that in \textbf{Step 2} of Appendix \ref{append:thmtotal}, we have shown the pseudo-contractivity for both $\T$ in \eqref{eq:T} and \eqref{eq:TATC}. In addition, although we do not assume \eqref{eq:initialfey}, the proof of \cite[Theorem 3.21]{Feyzmahdavian21} still holds, with the convergence rate \eqref{eq:linearfey} becomes
\begin{equation*}
    \|\bx^k-\bx^\star\|_{\infty}^b\le \rho^{\lfloor\frac{k}{B+D+1}\rfloor}\|\bx^0-\bx^\star\|_{\infty}^b.
\end{equation*}
Completes the proof.
\bibliographystyle{ieeetran}
\bibliography{reference}

\end{document}